\documentclass[10pt, oneside,reqno]{amsart}
\usepackage{amsmath, amssymb, amsfonts}
\usepackage{amsthm}
\usepackage{hyperref}

\usepackage{fancyhdr}
 
\theoremstyle{definition}
\newtheorem{definition}{Definition}[section]
\newtheorem{example}[definition]{Example}

\theoremstyle{plain}

\newtheorem{proposition}[definition]{Proposition}

\numberwithin{equation}{section}
\begin{document}
\author{Kamran Alam Khan}
\address{Department of Mathematics \\
 V. R. A. L. Govt. Girls P. G. College \\
 Bareilly (U.P.)-INDIA}
\email{kamran12341@yahoo.com}

\title{Generalized Fuzzy metric Spaces\\
 with an application to Colour image filtering}

\begin{abstract}
\noindent
Impulsive noise is a problem encountered during the acquisition and transmission of digital images. Fuzzy metrics dealing nicely with the nonlinear nature of digital images are used in vector median-based filters for noise reduction in colour and multichannel images. In this paper, We generalize the concept of Fuzzy metric space (In the sense of George and Veeramani) and introduce the notion of \emph{Generalized Fuzzy $n$-Metric Space}. The theory for such spaces is developed and as practical application, we propose some new filters based on these Generalized fuzzy metrics for colour image processing.
\end{abstract}
\maketitle
\section{Introduction}
 The notion of Fuzzy sets introduced by Zadeh \cite{33} paves a way to embrace upon vagueness and uncertainties arising in every day life . The Fuzzy theory has become an active area of research for the last Fifty years. The notion of Fuzzy metric evolved in two different perspectives. One group of researchers (e.g. \cite{8},\cite{11} ) following Kaleva \cite{12} uses fuzzy numbers to define metric as a non-negative real-valued function in ordinary spaces. The other group (e.g. \cite{5},\cite{7}) uses real numbers to measure the distance between fuzzy sets. Kramosil and Michalek \cite{15} generalized the concept of probabilistic metric space introduced by K. Menger \cite{27} to Fuzzy situations and introduced the notion of Fuzzy metric space. George and Veeramani (\cite{8},\cite{9},\cite{10}) modified this definition of fuzzy metric space with the help of continuous t-norms and proved that the topology in this new setting is Hausdorff. It is observed (\cite{17}) that fuzzy metrics are particularly useful in improving vector median-based filters for colour image filtering when replaces classical metrics . In a colour image, each pixel may be considered as an RGB component vector with integer values lying in the interval $[0,255]$ (\cite{22},\cite{23}). If $I_i=\Big(I_i(1),I_i(2),I_i(3)\Big)\in \{0,1,...,255\}^3$ denotes the colour image vector in the RGB colour space, Most commonly the nearness or closeness degree between two pixels $I_i$ and $I_j$ is described by the classical metric $L_p$, i.e.
 \begin{equation}
 L_p(I_i,I_j)=\Big(\sum_{l=1}^{3}|I_i(l)-I_j(l)|^p\Big)^\frac{1}{p}
 \end{equation}
 Morillas et al \cite{17} applied the concept of fuzzy metric (in the sense of George and Veeramani) to image filtering by proposing a variant of Vector Median Filter(VMF) that uses a fuzzy metric as distance criterion instead of the classical metrics used in VMF.
There are several generalizations of fuzzy metric spaces (e.g. \cite{28}, \cite{29}) for more than two variables. Recently Khan (\cite{13}, \cite{14}) attempted to generalize the notion of a $G$-metric space \cite{18} to more than three variables by introducing the notion of a $K$-metric, and later the notion of a generalized $n$-metric.  In this paper, We combine the concept of fuzzy metric space (in the sense of George and Veeramani)  with that of generalized $n$-metric space to introduce the notion of \emph{Generalized fuzzy $n$-metric space}. The theory for such spaces is developed and some new filters are proposed for colour image processing.

 \section{Preliminaries}
 
\begin{definition}
(\cite{14})
Let $X$ be a non-empty set, and $\mathbb{R}^+$ denote the set of non-negative real numbers. Let $G_n\colon X^n \to \mathbb{R}^+$, $(n\ge 3)$ be a function satisfying the following properties:
\begin{itemize}
\item [[G 1]] $G_n(x_1,x_2,...,x_n)=0$ if $x_1=x_2=\dots =x_n$,
\item [[G 2]] $G_n(x_1,x_1,...,x_1,x_2)>0$ for all $x_1$, $x_2 \in X$ with $x_1\neq x_2$,
\item [[G 3]] $G_n(x_1,x_1,...,x_1,x_2)\leq G_n(x_1,x_2,...,x_n)$ for all $x_1$, $x_2,... ,x_n \in X$ with the condition that any two of the points $x_2,\cdots ,x_n$ are distinct,
\item [[G 4]] $G_n(x_1,x_2,...,x_n)=G_n(x_{\pi(1)},x_{\pi (2)},...,x_{\pi(n)})$, for all $x_1$, $x_2,...,x_n \in X$ and every permutation $\pi$ of $\{1,2,...n\}$,  
\item [[G 5]] $G_n(x_1,x_2,...,x_n)\le G_n(x_1,x_{n+1},...,x_{n+1})+G_n(x_{n+1},x_2,...,x_n)$ for all $x_1$,$x_2,...,x_n,x_{n+1}\in X$. 
 
\end{itemize}
Then the function $G_n$ is called a \emph{Generalized $n$-metric} on $X$, and the pair $(X,G_n)$ a  \emph{Generalized n-metric space}.
\end{definition}
For $n=3$, $G_3(x_1,x_2,x_3)$ simply represents the $G$-metric space introduced by Mustafa and Sims \cite{18}. From now on we always have $n\ge 3$ for $(X,G_n)$ to be a generalized $n$-metric space.

\begin{example}
 Define a function $\rho\colon \mathbb{R}^n \to \mathbb{R}^+$,$(n\ge3)$ by
\begin{equation*}
\rho(x_1,x_2,\dots ,x_n)= \sum_{1\le r<s\le n}|x_r-x_s|
\end{equation*}
for all $x_1$, $x_2,...,x_n \in X$. Then $( \mathbb{R}, \rho)$ is a generalized $n$-metric space.
\end{example}
\begin{example}
\label{exmpl} For any metric space $(X,d)$, the following functions define generalized $n$-metrics on $X$:
\begin{itemize}
\item [(1)] $K_1^d(x_1,x_2,...,x_n)=\sum_{1\le r<s\le n} d(x_r,x_s)$,

\item [(2)] $K_2^d(x_1,x_2,...,x_n)=\text{max}\{d(x_r,x_s)\colon r,s\in \{1,2,...,n\},r<s\}$.
\end{itemize}

\end{example}
\begin{definition} \cite{8}
A binary operation $\ast: [0,1]\times [0,1]\to [0,1]$ is a continuous $t$-norm if it satisfies the following conditions:
\begin{itemize}
\item[(1)] $\ast$ is associative and commutative;
\item[(2)] $\ast$ is continuous;
\item[(3)] $a\ast 1=a$ for all $a\in [0,1]$;
\item[(4)] $a\ast b\le c\ast d$ whenever $a\le c$ and $b\le d$
\end{itemize}
for each $a,b,c,d\in [0,1]$.
\newline The examples of continuous $t$-norm are $a\ast b=ab$ and $a\ast b=\text{min} (a,b)$.
\end{definition}
Now We give the definition of fuzzy metric space introduced by George and Veeramani\cite{8} which constitutes a slight modification of the one given by Kramosil and Michalek \cite{15}.  
\begin{definition} \cite{8} \label{deffuzzy}A Fuzzy Metric Space is a triple  $(X, M, \ast )$ where $X$ is a nonempty set, $\ast$ is a continuous t-norm and $ M: X \times X \times (0,\infty) \to [0, 1]$ is a mapping (called fuzzy metric) which satisfies the following properties: for
every $x,y,z \in X$  and $s,t >0$ 
\begin{itemize}
\item [[FM 1]] $M(x, y, t) > 0$;
 
\item [[FM 2]] $ M(x, y, t) = 1$ if and only if $x = y$;
\item [[FM 3]] $M(x, y, t) = M(y, x, t)$;
\item [[FM 4]] $M(x, z, t + s) \ge M(x, y, t) \ast M(y, z, s)$  
\item [[FM 5]] $ M (x,y, \ast):  (0,\infty) \to (0, 1]$ is continuous. 
\end{itemize}
Then $M$ is called a fuzzy metric on $X$ and $M(x, y, t)$ denotes the degree of nearness between $x$ and $y$ with respect to $t$. \end{definition}
Throughout this paper, by a fuzzy metric $M(x,y,t)$ we always mean a fuzzy metric in the sense of George and Veeramani. 

\begin{example}  Let $X$ be a non-empty set and $d$ is a metric on $X$. Denote $a\ast b = a.b$ for all $a, b \in [0, 1]$. For each $t > 0$, define
\begin{equation*}
M_d(x,y,t)= \frac{t}{t+d(x,y)}
\end{equation*}

Then $(X,M_d, \ast)$ is a fuzzy metric space . We call this fuzzy metric $M_d$ induced by the metric $d$ the standard fuzzy metric.
\end{example}
\begin{definition} \cite{25}
The $3$-tuple $(X,N, \ast)$ is said to be a fuzzy normed space if $X$ is
a vector space, $\ast$ is a continuous $t$-norm and $N$ is a fuzzy set on $X\times(0,\infty)$
satisfying the following conditions for every $x, y \in X$ and $t, s > 0$:
\begin{itemize}
\item[[N 1]] $N(x, t) > 0$,
\item[[N 2]] $N(x, t) = 1$ iff $x = 0$,
\item[[N 3]] $ N(\alpha x,t)= N(x,\frac{t}{|\alpha |} )$, for all $\alpha \ne 0$,
\item[[N 4]] $N(x, t)\ast N(y, s)\le N(x + y, t + s)$,
\item[[N 5]] $N(x, .)\colon (0,\infty) \to [0, 1]$ is continuous,
\item[[N 6]] $\lim_{t\to \infty} N(x, t) = 1$.
\end{itemize}
\end{definition}

\begin{proposition}\cite{25} \label{fuzzynorm}
Let $(X,N,\ast)$ be a fuzzy normed space. If we define $M(x,y,t)=N(x-y,t)$, then $M$ is a fuzzy metric on $X$, which is called the fuzzy metric induced by the
fuzzy norm $N$.
\end{proposition}
 A topological space $(X, \tau )$ is said to be fuzzy metrizable if there exists a fuzzy metric $M$ on $X$ such that $\tau = \tau_M$. It was proved by George and Veeramani (\cite{8},\cite{9}) that every fuzzy metric $M$ on $X$ generates a topology $\tau_M$ on $X$. The family of open sets $\{B_M(x, r, t)\colon  x\in X,\; 0 < r < 1,\;t > 0\}$ forms a base for this topology, where $B_M(x, r, t) = \{y \in X : M(x, y, t) > 1-r\}$ for every $r, 0<r<1 $ and $t > 0$. This topological space is first countable and Hausdorff. Also for a metric space $(X, d)$, the topology generated by $d$ coincides with the topology $\tau_{M_d}$ generated by the standard fuzzy metric $M_d$. Thus every metrizable topological space is fuzzy metrizable. Gregori and Romaguera \cite{11} proved that if $(X,M, \ast)$ is a fuzzy metric space, then $\{U_n : n \in \mathbb{N}\}$ is a base for a uniformity $\mathcal{U}_M$ compatible with $\tau_M$, where $U_n = \{(x, y) \colon M(x, y, \frac{1}{n}) > 1-1/n\}$ for all $n \in \mathbb{N}$. Hence the topological space $(X, \tau_M)$ is metrizable.

\section{Main results}
Now we introduce the concept of \emph{Generalized Fuzzy $n$-metric space} as a generalization of the definition~\ref{deffuzzy}.
\begin{definition} \label{defgfuzzy} A 3-tuple $(X, F_n, \ast )$ is called \emph{Generalized Fuzzy $n$-metric space}
 if $X$ is an arbitrary (non-empty) set, $\ast$ is a continuous $t$-norm, and $F_n$ is a fuzzy set on $X^n \times (0,\infty)$ satisfying the following conditions for each $x_1,x_2,\dots x_n \in X$ and $t, s > 0$: 
\begin{itemize}
\item [[M 1]] $F_n(x_1,x_1,...,x_1,x_2,t)>0$ for all $x_1,x_2 \in X $ with $x_1\ne x_2$;
 
\item [[M 2]] $F_n(x_1,x_1,...,x_2,t) \ge F_n(x_1,x_2,...,x_n,t)$ with the condition that at least any two of the points  $x_2,x_3,\dots x_n$ are distinct;
\item [[M 3]] $F_n(x_1,x_2,...,x_n,t)=1$ if and only if $x_1=x_2= \dots =x_n$;

\item [[M 4]]  $F_n(x_1,x_2,...,x_n,t)=F_n(x_{\pi(1)},x_{\pi(2)},...,x_{\pi(n)},t)$  for every permutation $\pi$ of $\{1,2,...,n\}$;   
\item [[M 5]] $F_n(x_1,x_{n+1},...,x_{n+1},t)\ast F_n(x_{n+1},x_2,...,x_n,s)\le F_n(x_1,x_2,...,x_n,t+s)$; 
\item[[M 6]] $F_n(x_1,x_2,...,x_n,.)\colon (0,\infty) \to [0,1]$ is continuous.
\end{itemize}
\end{definition}
\begin{example}  Let $(X,G_n)$ be an Generalized $n$-metric space.  Denote $a\ast b = a.b$ for all $a, b \in [0, 1]$. For each $t > 0$, define
\begin{equation*}
F_n(x_1,x_2,...,x_n,t)= \frac{t}{t+G_n(x_1,x_2,...,x_n)}
\end{equation*}
for all $x_1,x_2,\dots x_n \in \mathbb{R}$. Then $(X, F_n,\ast)$ is a Generalized fuzzy $n$-metric space.
\end{example}
\begin{proposition} \label{propfuzzy1}
Let $(X, F_n,\ast)$ be a generalized fuzzy $n$-metric space. Then for $x,y\in X$ and $t>0$, we have
\begin{equation}
F_n(x,y,y,...y,t)\ge \Big[F_n\big(y,x,x,..x,\frac{t}{n-1}\big)\Big]^{n-1}
\end{equation}
\end{proposition}
\begin{proof}
The result is a direct consequence of [M 5] of definition~\ref{defgfuzzy}. 
\end{proof}
\begin{definition}
Let $(X, F_n,\ast)$ be a generalized fuzzy $n$-metric space. A subset $A$ of $X$ is said to be $F$-bounded if there exist $t>0$ and $r\in (0,1)$ such that
\begin{equation}
F_n(x_1,x_2,...,x_n,t)>1-r \; \text{for all} \; x_1,x_2,...,x_n\in A 
\end{equation}
\end{definition}
\begin{definition}
A generalized fuzzy $n$-metric $(F_n,\ast)$ on $X$ is said to be stationary if $F_n$ does not depend
on $t$, i.e. for each $x_1,x_2,...,x_n\in X$ the function $F_n(x_1,x_2,...,x_n,t)$ is constant.
\end{definition}

\begin{proposition}
 \label{propgfuzzy2}
Let $X$ be a closed real interval  $\left[a,b\right]$ and $K>\left|a\right|>0$. Consider for each $n=1,2,...$ the function ${F_r}^{(n)}\colon \underbrace{X^n\times X^n\times ... \times X^n}_r \times(0,\infty)\to (0,1]$ given by 
\begin{equation}
{F_r}^{(n)}(x^1,x^2,...,x^r,t)=\prod_{i=1}^n\frac{\min\{x^1_i,x^2_i,...,x^r_i\}+K }{\max\{x^1_i,x^2_i,...,x^r_i\}+K}
\end{equation}
Where $x^j=(x^j_1,x^j_2,...,x^j_n),\;\; j=1,2,..,r$ and $t>0$. Then $(X^n,{F_r}^{(n)},\ast)$ is a stationary $F$-bounded Generalized fuzzy $r$-metric space, where $a\ast b=a \cdot b$ for all $a,b\in [0,1]$. 
\end{proposition}
\begin{proof}
The Axioms [M 1]-[M 4] and [M 6] of the definition ~\ref{defgfuzzy} are obviously satisfied for ${F_r}^{(n)}$. We use mathematical induction to prove the triangle inequality [M 5].
For $n=1$, We have $x^j=x^j_1$ and
\begin{equation*}
{F_r}^{(1)}(x^1,x^2,...,x^r,t+s)=\frac{\min\{x^1,x^2,...,x^r\}+K }{\max\{x^1,x^2,...,x^r\}+K}
\end{equation*}
Let $x^h\in X$, then many different cases arise

i.e. $x^1\le x^2\le ...\le x^r\le x^h$;    $x^1\le x^h \le x^2 \le ... \le x^r$;  $x^h\le x^2\le ...\le x^r\le x^1$;   etc.
 
 For the case $x^1\le x^2\le ...\le x^r\le x^h$, we have
 \begin{equation*}
 \frac{x^1 + K}{x^r + K}\ge \frac{x^1 +K}{x^h + K}\cdot \frac{x^2 + K}{x^h + K}
 \end{equation*}
 Hence the inequality ${F_r}^{(1)}(x^1,x^2,...,x^r,t+s)\ge {F_r}^{(1)}(x^1,x^h,...,x^h,t)\cdot {F_r}^{(1)}(x^h,x^2,...,x^r,s)$ holds. Similar argument verifies the inequality for other cases.
  
 Now suppose the triangle inequality holds for $n=m-1$. Then for each $t,s>0$ we have
\begin{align*}
\begin{split}
{F_r}^{(m)}(x^1,x^2,...,x^r,t+s)&=\prod_{i=1}^m\frac{\min\{x^1_i,x^2_i,...,x^r_i\}+K }{\max\{x^1_i,x^2_i,...,x^r_i\}+K}\\
&=\prod_{i=1}^{m-1}\frac{\min\{x^1_i,x^2_i,...,x^r_i\}+K }{\max\{x^1_i,x^2_i,...,x^r_i\}+K} \cdot \frac{\min\{x^1_m,x^2_m,...,x^r_m\}+K }{\max\{x^1_m,x^2_m,...,x^r_m\}+K}\\
&\ge  \prod_{i=1}^{m-1}\frac{\min\{x^1_i,x^h_i,...,x^h_i\}+K }{\max\{x^1_i,x^h_i,...,x^h_i\}+K}
 \cdot \prod_{i=1}^{m-1}\frac{\min\{x^h_i,x^2_i,...,x^r_i\}+K }{\max\{x^h_i,x^2_i,...,x^r_i\}+K} \\
 & \qquad \cdot \frac{\min\{x^1_m,x^h_m,...,x^h_m\}+K }{\max\{x^1_m,x^h_m,...,x^h_m\}+K}\cdot \frac{\min\{x^h_m,x^2_m,...,x^r_m\}+K }{\max\{x^h_m,x^2_m,...,x^r_m\}+K}\\
 &=\prod_{i=1}^{m}\frac{\min\{x^1_i,x^h_i,...,x^h_i\}+K }{\max\{x^1_i,x^h_i,...,x^h_i\}+K}
  \cdot \prod_{i=1}^{m}\frac{\min\{x^h_i,x^2_i,...,x^r_i\}+K }{\max\{x^h_i,x^2_i,...,x^r_i\}+K} \\
  &={F_r}^{(m)}(x^1,x^h,...,x^h,t)\cdot {F_r}^{(m)}(x^h,x^2,...,x^r,s)
\end{split}
\end{align*}
 Thus $(X^n,{F_r}^{(n)},\cdot)$ is a Generalized fuzzy $r$-metric space for all $n(=1,2,...)$. Since it is independent of $t>0$, hence it is stationary.
 
 Now $a\le x^j_i\le b$ for all $1\le i\le n$ and $1\le j\le r$, hence for all $x^1,x^2,...,x^r\in X^n$ and $t>0$, we have $\min\{x^1_i,x^2_i,...x^r_i\}\ge a$, $\max\{x^1_i,x^2_i,...x^r_i\}\le b$ and therefore
 \begin{equation*}
 {F_r}^{(n)}(x^1,x^2,...,x^r,t)\ge \Big(\frac{a+K}{b+K}\Big)^n>0
 \end{equation*}
i.e. $X^n$ is $F$-bounded.
    
\end{proof}

\begin{definition} Let $(X, F_n,\ast)$ be a Generalized fuzzy $n$-metric space.For $t>0$, the open ball $B_F(x_0,r,t)$ with center $x_0$ and radius $0<r<1$ is defined by
\begin{equation}
B_F(x_0,r,t)=\{y\in X \colon F_n(x_0,y,y,...,y,t)>1-r\}
\end{equation}
\end{definition}
\begin{definition} A subset $A$ of $X$ is called an open set if for each $x\in A$ there exist $t>0$ and $0<r<1$ such that $B_F(x,r,t)\subset A$.
\end{definition}
\begin{proposition}
Let $(X, F_n,\ast)$ be a generalized fuzzy $n$-metric space. Define 
$\tau_F =\{A \subset X \colon  x\in A\;  \text{if and only if there exists}\; t>0,\text{and}\;  r,\,0<r<1,\text{such that}\; B_F(x,r,t)\subset A\}$. Then $\tau_F$ is a topology on $X$.
\end{proposition}
\begin{proof}
It is straightforward.
\end{proof}
\begin{proposition}
Let $(X, F_n,\ast)$ be a generalized fuzzy $n$-metric space. Then $(X,\tau_F)$ is Hausdorff.
\end{proposition}
\begin{proof}
Let $x, y$ be two distinct points of $X$. Then by [M 1], we have $0 < F_n(x,y, \dots y,t) < 1$. Let $F_n(x,y, \dots y,t)=r$ for some $r; 0 <
r < 1$. For each $r_0; r < r_0 < 1$, we can find an $r_1$ such that
\begin{equation*}
{\underbrace{r_1\ast r_1\ast \dots \ast r_1}_n}\ge r_0
\end{equation*} 
 Let us consider  the open balls $B_F(x,1-r_1,\frac{t}{n})$ and $B_F(y,1-r_1,\frac{t}{n})$.
We claim that these balls are disjoint. For if $z\in B_F(x,1-r_1,\frac{t}{n})\cap B_F(y,1-r_1,\frac{t}{n})$, then $ F_n(x,z,\dots z,\frac{t}{n})>r_1 $ and $ F_n(y,z,\dots z,\frac{t}{n})>r_1 $ and we have
\begin{align*}
\begin{split}
r=F_n(x,y,\dots y,t)&\ge F_n(x,z,\dots z,\frac{t}{n})\ast F_n(y,z,\dots z,\frac{t}{n})\ast \dots \ast F_n(y,z,\dots z,\frac{t}{n})\\
&>r_1\ast r_1 \ast \dots \ast r_1\ge r_0>r 
\end{split}
\end{align*}
which is a contradiction. Therefore $(X,\tau_F)$ is Hausdorff.
\end{proof}

\begin{definition} Let $(X, F_m,\ast)$ be Generalized fuzzy $m$-metric space. A sequence $<x_n>$ in $X$ is said to be $F_m$-convergent if and only if there exists $x\in X$ such that $F_m(x_n, x_n,...,x_n,x,t)\to 1$ as $n\to \infty$, for each $t>0$. Or equivalently, $F_m(x, x,...,x,x_n,t)\to 1$ as $n\to \infty$, for each $t>0$.
\end{definition}
\begin{definition} Let $(X, F_m,\ast)$ be Generalized fuzzy $m$-metric space. A sequence $<x_n>$ in $X$ is $F_m$-Cauchy if and only if for every $\epsilon >0$ there exists $K\in \mathbb{N}$ such that
\begin{equation}
F_m(x_{n_1}, x_{n_2},...,x_{n_m},t)> 1-\epsilon \; \text{for all} \; n_1,n_2,...,n_m\ge K
\end{equation}
\end{definition}
\begin{definition} A generalized fuzzy $n$-metric space $(X, F_n,\ast)$ is said to be \emph{$F_n$-complete} if every $F_n$-Cauchy sequence in $X$ is $F_n$-convergent in $X$.
\end{definition}

\begin{proposition} Let $(X, F_n,\ast)$ be Generalized fuzzy $n$-metric space. Then $F_n(x_1, x_2,...,x_n,t)$ is
nondecreasing with respect to t, for all $x_1,x_2,\dots ,x_n$ in $X$ and we have
\begin{equation*}
\lim_{t\to \infty}F_n(x_1,x_2,\dots x_n,t)=1
\end{equation*}
\end{proposition}
\begin{proof}
From [M 5] of definition ~\ref{defgfuzzy}, We have 
\begin{equation*}
F_n(x_1,x_{n+1},...,x_{n+1},s)\ast F_n(x_{n+1},x_2,...,x_n,t)\le F_n(x_1,x_2,...,x_n,s+t)
\end{equation*}

Taking $x_{n+1}=x_1$, We have
\begin{equation*}
F_n(x_1,x_1,...,x_1,s)\ast F_n(x_1,x_2,...,x_n,t)\le F_n(x_1,x_2,...,x_n,t+s)
\end{equation*}
\begin{equation*}
\Rightarrow F_n(x_1,x_2,...,x_n,t)\le F_n(x_1,x_2,...,x_n,t+s)
\end{equation*}
and therefore $\lim_{t\to \infty}F_n(x_1,x_2,\dots x_n,t)=1$.
\end{proof}
   
 \begin{proposition} \label{fuzzyprop} Let $(X,F, \ast)$ be a fuzzy metric space. If we define $F_n:X^n\times (0,\infty) \to [0,1]$ by
 \begin{equation}
 \label{fuzcond2}
 F_n(x_1,x_2,\dots x_n,t)=\prod_{1\le i<j\le n} F(x_i,x_j,t)
 \end{equation}
 for every  $x_1,x_2,\dots, x_n\in X$ then $(X,F_n,\ast)$ is a generalized fuzzy n-metric metric space and we have
 \begin{equation}
 \label{fuzcond3}
 [F_n(x_1,x_2,\dots x_n,t)]^{n-2}=\prod_{1\le i_1<i_2<\dots <i_{n-1}\le n} F_{n-1} (x_{i_1},x_{i_2},\dots x_{i_{n-1}},t)
 \end{equation}
 \end{proposition}
 \begin{proof}
 It is trivial to show that $F_n(x_1,x_2,\dots x_n,t)$ defined by equation ~\ref{fuzcond2} satisfies all the axioms but [M 5] of Generalized fuzzy n-metric. We use induction to show that [M 5] holds. For $n=3$ we have
 \begin{align*}
 \begin{split}
 F_3(x_1,x_2,x_3,t)&=\prod_{1\le i<j\le 3} F(x_i,x_j,t+s)\\
                   &=F(x_1,x_2,t+s)\ast F(x_1,x_3,t+s) \ast F(x_2,x_3,t+s)\\
                   &\ge F(x_1,x_4,t)\ast F(x_4,x_2,s)\ast F(x_1,x_2,t)\ast F(x_2,x_3,s)\ast F(x_2,x_4,t)\ast F(x_4,x_3,s)\\
                   &=F_3(x_1,x_2,x_4,t)\ast F_3(x_2,x_3,x_4,s)\\
                   &\ge F_3(x_1,x_4,x_4,t)\ast F_3(x_2,x_3,x_4,s)\\
 \end{split}
 \end{align*} 
 Let the condition [M 5] holds for all $n$ up to $n=k$. For $n=k+1$ we have
  \begin{align*}
  \begin{split}
  F_{k+1}(x_1,x_2,\dots x_{k+1},t+s)&=\prod_{1\le i<j\le {k+1}} F(x_i,x_j,t+s)\\
  &=F_k(x_1,x_2,\dots x_k,t+s)\ast \prod_{1\le i\le k} F(x_i,x_{k+1},t+s)\\
  &\ge F_k(x_1,x_{k+2},\dots x_{k+2},t)\ast F_k(x_{k+2},x_2,\dots x_k,s)\ast F(x_1,x_{k+1},t+s)\\
  &\qquad \ast  \prod_{2\le i\le k} F(x_i,x_{k+1},t+s)\\
  &\ge [F(x_1,x_{k+2},t)]^{k-1}\ast F_k(x_{k+2},x_2,\dots x_k,s)\ast F(x_1,x_{k+2},t) \\
  &\qquad \ast F(x_{k+2},x_{k+1},s) \ast  \prod_{2\le i\le k} F(x_i,x_{k+1},t+s)\\ 
  &\ge [F(x_1,x_{k+2},t)]^k\ast F_k(x_{k+2},x_2,\dots x_k,s) \ast F(x_{k+2},x_{k+1},s)\\
    &\qquad  \ast  \prod_{2\le i\le k} F(x_i,x_{k+1},s) \qquad \text{as} \quad F(x_i,x_{k+1},t+s) \ge F(x_i,x_{k+1},s)\\
    &= F_{k+1}(x_1,x_{k+2}\dots x_{k+2},t)\ast F_{k+1}(x_{k+2},x_2,\dots x_{k+1},s)
  \end{split}
  \end{align*}
 Thus $F_n(x_1,x_2,\dots x_n,t)$ defined by ~\ref{fuzcond2} is a Generalized fuzzy n-metric.
 \newline Obviously $F_2(x_i,x_j,t)=F(x_i,x_j,t)$. Hence from ~\ref{fuzcond2} it follows that
 \begin{equation*}
\prod_{1\le i_1<i_2<\dots <i_{n-1}\le n} F_{n-1} (x_{i_1},x_{i_2},\dots x_{i_{n-1}},t)=\prod_{1\le i_1<i_2<\dots <i_{n-1}\le n}  \prod_{\substack{i<j\\(i,j)\in \{i_1,i_2\dots i_{n-1}\}}} F_2(x_i,x_j,t)
 \end{equation*}
 Since each pair $(i,j), i<j$ is in exactly $(n-2)$ sequences $\{i_1,i_2\dots i_{n-1}\},1\le i_1<i_2<\dots <i_{n-1}\le n$, hence
  \begin{equation*}
 \prod_{1\le i_1<i_2<\dots <i_{n-1}\le n} F_{n-1} (x_{i_1},x_{i_2},\dots x_{i_{n-1}},t)=[\prod_{1\le i <j\le n}  F_2(x_i,x_j,t)]^{n-2}=[F_n(x_1,x_2,\dots x_n,t)]^{n-2}
  \end{equation*}
  \end{proof}
   
 From above proposition, it is clear that for a given fuzzy metric $F_2 (x,y,t)$ we can always define a generalized fuzzy $n$-metric given by equation ~\ref{fuzcond2}. The converse is also true-
 \begin{proposition}
  If $F_n$ is a generalized fuzzy $n$-metric defined on $X$ then there exists a fuzzy metric $M$ on $X$ given by
  \begin{equation} \label{fuzzyeq1}
  M(x,y,t)=F_n(x,y,y,...,y,\frac{t}{2})\ast F_n(x,x,...x,y,\frac{t}{2})
 \end{equation}
 \end{proposition}
\begin{proof}
 It is easy to verify that $M$ is indeed a fuzzy metric (In the sense of George and Veeramani) defined on $X$.
\end{proof}
 \begin{proposition}
 Let $B_M(x,r,t)$ denote the open ball in the fuzzy metric space $(X,M,\ast)$ and $B_F(x,r,t)$ the open ball in the corresponding generalized fuzzy $n$-metric space $(X,F_n,\ast)$. Then for $0<r<1$ and $t>0$ there exists $0<s<1$ such that $B_F(x,\frac{r}{n},\frac{t}{n-1})\subseteq B_M(x,s,2t)$.
 \begin{proof}
 Let $y\in B_F(x,\frac{r}{n},\frac{t}{n-1})$, then $F_n(x,y,y,...,y,\frac{t}{n-1})>1-\frac{r}{n}$. By proposition ~\ref{propfuzzy1} and [M 4] we have
 \begin{align*}
 \begin{split}
 F_n(x,x,...,x,y,t)&\ge \Big[F_n\Big(x,y,y,...,y,\frac{t}{n-1}\Big)\Big]^{n-1}\\
 &>\Big (1-\frac{r}{n}\Big)^{n-1}
 \end{split}
 \end{align*}
 Since $F_n$ is nondecreasing, hence
 \begin{equation*}
 F_n(x,y,y,...,y,t)\ge F_n\Big(x,y,y,...,y,\frac{t}{n-1}\Big)>1-\frac{r}{n}
 \end{equation*}
 Using equation ~\ref{fuzzyeq1}, we have
 \begin{align*}
 \begin{split}
 M(x,y,2t)&=F_n(x,y,...,y,t)\ast F_n(x,x,...,x,y,t)\\
 &>\Big (1-\frac{r}{n}\Big)^{n-1}\ast \Big(1-\frac{r}{n}\Big)=\Big (1-\frac{r}{n}\Big)^{n} 
 \end{split}
 \end{align*}
 Now $0<\frac{r}{n}<1$, hence there exists $0<s<1$ such that 
 \begin{equation*}
 \underbrace{\Big(1-\frac{r}{n}\Big)\ast \Big(1-\frac{r}{n}\Big)\ast ...\ast \Big(1-\frac{r}{n}\Big)}_n>1-s
 \end{equation*}
 Thus $M(x,y,2t)>1-s$, i.e. $y\in B_M(x,s,2t)$. Hence the result.
 \end{proof}
 
 \end{proposition}
  Therefore the topology $\tau_F$ induced by the generalized $n$-fuzzy metric on $X$ coincides with the topology $\tau_M$ induced by the fuzzy metric $M$ implying that every generalized fuzzy $n$-metric space is topologically equivalent to a fuzzy metric space.
  Since the topological space $(X, \tau_M)$ is metrizable \cite{11},  $(X,\tau_F)$ is also a metrizable space.

   It is observed that fuzzy metrics are useful in improving filters for colour image filtering. In the following, we discuss the relevant work, some definitions and our proposals to use generalized fuzzy metrics for constructing new possible filters.
   
  \section{Applications to colour image Filtering} Any image gets affected due to the noise introduced during the acquisition and transmission process. The presence of noise corrupts the visualization quality of image and affects the image processing steps such as pattern recognition, image segmentation etc. Therefore, The elimination or suppression of noise, i.e. the image filtering is an essential part in any computer vision system. The vector approach (\cite{2},\cite{16},\cite{22}) is more appropriate for the elimination of noise compared to other approaches. In this approach, each pixel value is considered as an $m$-dimensional vector. where $m$ is the number of image channels (for colour images $m=3$). The operation of most of the nonlinear filters is based on robust order statistics (\cite{3},\cite{21},), i.e. Marginal ordering (M-ordering), Conditional ordering (C-ordering), Partial ordering (P-ordering) and Reduced or aggregated ordering (R-ordering). The colour images are treated as a vector field \cite{16} and the filter selects the output vector on the basis of the ordering of vectors in a defined sliding window. This window is moved over the input image affecting all the image pixels. 
  
   Let $W$ be a processing window of size $n$ and let $I_j, j=1,2,..,n$ be the noisy image vectors(pixels) in the filtering window. To order colour vectors $I_1,I_2,...,I_n$ located inside the window $W$, the R-ordering based vector filters use the aggregated distances, i.e. $D^i=\sum_{j=1}^{n} d(I_i,I_j)$ or the aggregated similarities, i.e. $S^i=\sum_{j=1}^{n} s(I_i,I_j)$ associated with the input vector $I_i$ for $i=1,2,...,n$. Then the ordered sequence of aggregated distances $D^i: D^{(1)}\le D^{(2)}\le ...\le D^{(n)}$ implies the same ordering of the corresponding vectors $I_i: I_{(1)}\le I_{(2)},...\le I_{(n)}$. Most of the filters use the lowest ranked vector $I_{(1)}$ as the output due to the fact that the higher indexed vectors in the ordered sequence diverge greatly from the data population. Nonlinear multichannel filters \cite{20} utilizes various distance measures for image filtering purposes. The most well known representative of this class of filters is the vector median filter(VMF)(\cite{2},\cite{20}). 
   
 The classical vector filters often tend to blur image edges and details. Recently a number of fuzzy based methods (\cite{6},\cite{17},\cite{26},\cite{31},\cite{32}) have been proposed  to address these drawbacks. In an attempt to rectify the performance of classical vector filters, Morillas et al \cite{17} used a special fuzzy metric as distance criterion instead of the classical metrics used in VMF. They defined the fuzzy distance between two pixel vectors $I_i=\Big(I_i(1),I_i(2),I_i(3)\Big)$ and $I_j=\Big(I_j(1),I_j(2),I_j(3)\Big)$ as  
\begin{equation}
 M(I_i,I_j)=\prod_{l=1}^3 \frac{\min\{I_i(l),I_j(l)\}+K}{\max\{I_i(l),I_j(l)\}+K}
\end{equation}
Where $I_i(l),I_j(l)\in \text{integer values in}\;[0,255]$ for the processing of RGB images and $K>0$. According to \cite{17} the appropriate value of $K$ for RGB colour vectors is $1024$. It can be verified that $(M,\ast)$ is a stationary $F$-bounded fuzzy metric $M(I_i,I_j)$ on $X^3$, Where $X$ is a closed real interval $[a,b]$ and $a\ast b=a\cdot b$. In their proposal, the scalar quantity $M^i=\sum_{j=1,j\ne i}^{n} M(I_i,I_j)$ is the accumulated fuzzy distance associated with the vector $I_i$. The ordering of $M^i$ with fuzzy rules(\cite{1},\cite{4},\cite{17}) is now defined as $M^i:M^{(1)}\ge M^{(2)}\ge ... \ge M^{(n)}$. Which implies the ordering of vectors$I_i$ as $I_i:I_{(1)}\ge I_{(2)}\ge ...\ge I_{(n)}$. Finally the output vector is defined as $I_{(1)}$.

We propose to replace the above fuzzy metric by a more general fuzzy metric ${F_r}^{(3)}(I_1,I_2,...,I_r,t)$ defined in proposition ~\ref{propgfuzzy2}. Since ${F_r}^{(3)}$ is stationary, We simply denote it as ${F_r}^{(3)}(I_1,I_2,...,I_r)$. Now We define the accumulated fuzzy measure associated to the vector $I_i$ as
\begin{equation}\label{eqgfuzzy3}
    D^i=\sum_{\substack{1\le i_1<i_2<...<i_r\le n\\i_k\ne i,\, k=1,2,...,r-1}} {F_r}^{(3)}(I_i,I_{i_1},...,I_{i_{r-1}})
\end{equation}
Here We have more choices in terms of different values of $r$. The value of $r$ can be chosen so as to make the algorithm computationally more efficient.  For $r=2$ the metric ${F_2}^{(3)}(I_i,I_{i_1})$ coincides with the fuzzy metric $M(I_i,I_{i_1})$ employed by Morillas et al. The fuzzy metric $M(I_i,I_j)$ can be thought of as the degree of nearness between the pixels represented by the vectors $I_i$ and $I_j$. Similarly the metric ${F_r}^{(3)}(I_{i_1},...,I_{i_r})$ expresses some type of nearness of the pixels $I_{i_1},...,I_{i_r}$.
    
One way to use equation ~\eqref{eqgfuzzy3} efficiently is to choose $r=3$. Then the fuzzy metric ${F_3}^{(3)}(I_{i_1},I_{i_2},I_{i_3})$ will be a particular stationary form of a more generalized fuzzy metric $F_3(x,y,z,t)$. This fuzzy metric can be induced by a $G$-metric $G(x,y,z)$ introduced by Mustafa and Sims \cite{18}. One interpretation of such metric could be that $F_3(x,y,z,t)=\alpha$ if and only if the probability 
     $P[\text{(perimeter of the triangle with
      vertices}\\
       \; x,y\;  \text{and}\; z)\le t]=\alpha$.
     
Thus we can propose a new filter by taking fuzzy analogue of perimeter of a triangle formed by three pixels instead of fuzzy distance between two pixels. Then $D^i$ will represent the sum of fuzzy analogue perimeters of all the triangles formed with the pixel $I_i$ as one vertex. So an appropriate fuzzy metric for our new filter could be proposed as  
\begin{equation} \label{eqgfuzzy4}
   {F_3}^{(3)}(I_{i_1},I_{i_2},I_{i_3})=\prod_{l=1}^3\frac{\min\{I_{i_1}(l),I_{i_2}(l),I_{i_3}(l)\}+K }{\max\{I_{i_1}(l),I_{i_2}(l),I_{i_3}(l)\}+K}
\end{equation}
Where $I_{i_p}=\Big(I_{i_p}(1),I_{i_p}(2),I_{i_p}(3)\Big)\in \text{integer values in}\; [0,255],\; p=1,2,3$ for the processing of RGB images and $K>0$. The non-uniformity of the fuzzy metric ${F_3}^{(3)}$ can be smoothed by taking large values of $K$. However if $K\rightarrow \infty$ then ${F_3}^{(3)}\to 1$, hence the use of very large values should be avoided. We can decide an appropriate value of $K$ by analyzing the performance(MSE) of the fuzzy metric ${F_3}^{(3)}$ with respect to different values of $K$. 
Then the accumulated fuzzy measure associated to the vector $I_i$ as
\begin{equation}
        D^i=\sum_{\substack{1\le i_1<i_2\le n\\i_1,I_2\ne i}} {F_3}^{(3)}(I_i,I_{i_1},I_{i_2})
\end{equation}
We can avoid the repetitive effect of some extra terms in the sum $D^i$ by following an appropriate algorithm. For example, in a 8-neighbourhood $3\times 3$ window if the pixels $(1,1),(1,2),(1,3),(2,1)...,(3,3)$ are numbered as $1,2,3,..,9$ respectively,then the pixels numbered $1,3,7$ and $9$ will have similar spatial neighborhoods. Similarly pixels numbered $2,4,6$ and $8$ are spatially similar. The central pixel is that numbered $5$. We can use the following scheme to evaluate $D^1$, $D^2$ and $D^5$ avoiding the terms contributing the repetition.
 
$D^1={F_3}^{(3)}(I_1,I_2,I_4)+{F_3}^{(3)}(I_1,I_3,I_7)+{F_3}^{(3)}(I_1,I_6,I_8)+{F_3}^{(3)}(I_1,I_5,I_9)$,\\

$D^2={F_3}^{(3)}(I_2,I_1,I_3)+{F_3}^{(3)}(I_2,I_4,I_6)+{F_3}^{(3)}(I_2,I_7,I_9)+{F_3}^{(3)}(I_2,I_5,I_8)$\\

$D^5={F_3}^{(3)}(I_5,I_1,I_9)+{F_3}^{(3)}(I_5,I_2,I_4)+{F_3}^{(3)}(I_5,I_3,I_7)+{F_3}^{(3)}(I_5,I_6,I_8)$

The other accumulated fuzzy measures i.e. $D^3,D^4,...$ can be computed using a similar approach.

Since the deviation of $D^i$ from the accumulated fuzzy distance $\sum_{j=1,j\ne i}^n M(I_i,I_j)$ is not very large for usual window size (n=9), hence the vector order statistics techniques for a VMF(fuzzy) proposed in \cite{17} can be used here. Therefore the filter output for our proposed filter will be $\tilde{I_k}\in W$ that maximizes the aggregated fuzzy measure to the other vectors in $W$, i.e.
$I_\text{out}=\tilde{I_k}$ for which
\begin{equation*}
k=\arg \max_i (D^i),\quad i=1,2,...,n
\end{equation*}

This filter is not a direct generalization of VMF and therefore cannot be termed as some kind of median filter. However it shares some commonalities with VMF such as the use of same vector order\textendash statistics and the use of same kind of fuzzy metric, hence We may call it  \emph{Fuzzy Vector Median-Like Filter} (FVMLF).
 \section{Future Directions}
 \begin{itemize}
\item[(1)]
There could be other more efficient ways to choose the value of $r$ in equation ~\ref{eqgfuzzy3}. For example, If we choose $(r=n-1)$, the following scheme can be analyzed for computational complexity and performance analysis\textendash
\begin{equation*}
 D^k=
\begin{cases}
 {F^{(3)}_{n-1}}(I_1,I_2,...,I_{n-1})+{F^{(3)}_{n-1}}(I_1, I_3,...,I_{n}) \qquad \text{if } k=1,\\
 
 {F^{(3)}_{n-1}}(I_k, I_1,...,I_{k-2},I_{k+1},...,I_{n-1},I_n)+{F^{(3)}_{n-1}}(I_k, I_1,...,I_{k-1},I_{k+2},...,I_{n}),\\
  \qquad \qquad \qquad \qquad \qquad \qquad \qquad \qquad \qquad \qquad \qquad \text{if } 1<k<n-1,\\
 {F^{(3)}_{n-1}}(I_{n-1},I_1,...,I_{n-3},I_n)+{F^{(3)}_{n-1}}(I_{n-1},I_1,...,I_{n-2}) \qquad \text{if } k=n-1,\\
 {F^{(3)}_{n-1}}(I_n,I_1,...,I_{n-2})+{F^{(3)}_{n-1}}(I_n,I_2,...,I_{n-1}) \qquad \text{if } k=n.
 \end{cases}
 \end{equation*}

It would be interesting to see whether we could still use the fuzzy rule based order statistics (\cite{1},\cite{17}) here.
\item[(2)]In above proposals, our approach is inspired by the VMF introduced in \cite{17}. However, The equation ~\ref{eqgfuzzy3} is a more general relation giving space for using generalized fuzzy metrics ${F_r}^{(3)}(I_i,I_{i_1},...,I_{i_{r-1}})$ other than of the form given in equation ~\ref{eqgfuzzy4}. We can try some other accumulated fuzzy measures associated to the vector $I_i$ similar to the proposed one. 
\item[(3)] We also see a possibility to construct a \emph{Geometric mean filter} using the concept of fuzzy norm and the generalized fuzzy metric. A geometric mean filter replaces the colour value of each pixel with the geometric mean of colour pixel values of the neighboring cells (including the centered pixel) of the filtering window. Let $I_1,I_2,...,I_p$ be the noisy image vectors (pixels) in a $p (=k\times k)$ size filtering window centered at $I_j$. Now in view of propositions ~\ref{fuzzynorm} and ~\ref{fuzzyprop}, the generalized fuzzy $n$-metric can be given as
\begin{equation}\label{eqgfuzzy5}
   F_n(x_1,x_2,\dots x_n,t)=\prod_{1\le i<j\le n} M(x_i,x_j,t)=\prod_{1\le i<j\le n} N(x_i-x_j,t)
\end{equation}
Equation~\ref{eqgfuzzy5} gives us a way to define some fuzzy norm $N(I_i,t)$ associated with each vector $I_i, i=1,2,...,p$ and then we have a Geometric mean filter which replaces the vector $I_j$ by the vector $\tilde{I_j}$ such that
\begin{equation}
N(\tilde{I_j},t)=\Big[\prod_{1\le i\le p}N(I_i,t)\Big]^{1/p}
\end{equation}
 \end{itemize}
From practical application point of view, the performance of our proposed filters can be evaluated by using standard measures like Mean Absolute Error(MAE), Peak Signal to Noise Ratio(PSNR) and Normalized Colour Difference(NCD) and then the proposed filters can be included within more complex filtering procedures to process the most noisy regions of the colour images or to reduce the impulse noise. we do hope that this study will provide the basis of much future research.

\section*{Acknowledgments}
We are thankful to the referees for their valuable comments and suggestions on this manuscript.

\providecommand{\bysame}{\leavevmode\hbox
to3em{\hrulefill}\thinspace}


\end{document}